\documentclass[10pt]{amsart}
\usepackage{amsmath}
\usepackage{amsxtra}
\usepackage{amscd}
\usepackage{amsthm}
\usepackage{amsfonts}
\usepackage{amssymb}
\usepackage{eucal}

\usepackage{latexsym}

\newtheorem{cor}[subsection]{Corollary}
\newtheorem{lem}[subsection]{Lemma}
\newtheorem{prop}[subsection]{Proposition}

\newtheorem{thm}[subsection]{Theorem}

\theoremstyle{remark}

\theoremstyle{definition}

\numberwithin{equation}{section}

\newcommand{\thmref}[1]{Theorem~\ref{#1}}

\newcommand{\lemref}[1]{Lemma~\ref{#1}}
\newcommand{\propref}[1]{Proposition~\ref{#1}}

\newcommand{\nc}{\newcommand}
\nc{\renc}{\renewcommand}
\nc{\ssec}{\subsection}
\nc{\sssec}{\subsubsection}
\nc{\on}{\operatorname}

\nc\ol{\overline}
\nc\ul{\underline}
\nc\wt{\widetilde}
\nc\tboxtimes{\wt{\boxtimes}}
\nc{\wh}{\widehat}
\nc{\mc}{\mathcal}

\nc{\CM}{{\mathcal M}}
\nc{\CN}{{\mathcal N}}
\nc{\CF}{{\mathcal F}}
\nc{\D}{{\mathcal D}}
\nc{\CQ}{{\mathcal Q}}
\nc{\CY}{{\mathcal Y}}
\nc{\CX}{{\mathcal X}}
\nc{\CG}{{\mathcal G}}
\nc{\CE}{{\mathcal E}}
\nc{\CC}{{\mathcal C}}
\nc{\CO}{{\mathcal O}}
\renc{\CC}{{\mathcal C}}
\nc{\CT}{{\mathcal T}}
\nc{\CK}{{\mathcal K}}
\nc{\CS}{{\mathcal S}}
\nc{\CH}{{\mathcal H}}
\nc{\CU}{{\mathcal U}}
\nc{\CV}{{\mathcal V}}
\nc{\CA}{{\mathcal A}}
\nc{\CB}{{\mathcal B}}
\nc{\CW}{{\mathcal W}}
\nc{\CL}{{\mathcal L}}
\nc{\CP}{{\mathcal P}}
\nc{\CI}{{\mathcal I}}
\nc{\CJ}{{\mathcal J}}
\nc{\CR}{{\mathcal R}}
\nc{\CZ}{{\mathcal Z}}

\nc{\BA}{{\mathbb{A}}}
\nc{\BC}{{\mathbb{C}}}
\nc{\BG}{{\mathbb{G}}}
\nc{\BM}{{\mathbb{M}}}
\nc{\BN}{{\mathbb{N}}}
\nc{\BP}{{\mathbb{P}}}
\nc{\BR}{{\mathbb{R}}}
\nc{\BZ}{{\mathbb{Z}}}
\nc{\BV}{{\mathbb{V}}}
\nc{\BW}{{\mathbb{W}}}
\nc{\BS}{{\mathbb{S}}}
\nc{\BD}{{\mathbb{D}}}
\nc{\BQ}{{\mathbb{Q}}}
\nc{\BL}{{\mathbb{L}}}
\renc{\BW}{{\mathbb{W}}}

\nc{\fa}{{\mathfrak{a}}}
\nc{\fb}{{\mathfrak{b}}}
\nc{\fg}{{\mathfrak{g}}}
\nc{\fgl}{{\mathfrak{gl}}}
\nc{\fh}{{\mathfrak{h}}}
\nc{\fj}{{\mathfrak{j}}}
\nc{\fm}{{\mathfrak{m}}}
\nc{\fl}{{\mathfrak{l}}}
\nc{\fn}{{\mathfrak{n}}}
\nc{\fu}{{\mathfrak{u}}}
\nc{\fp}{{\mathfrak{p}}}
\nc{\ff}{{\mathfrak{f}}}
\nc{\fd}{{\mathfrak{d}}}
\nc{\fr}{{\mathfrak{r}}}
\nc{\fs}{{\mathfrak{s}}}
\nc{\fsl}{{\mathfrak{sl}}}

\nc{\hsl}{{\widehat{\mathfrak{sl}}}}
\nc{\hgl}{{\widehat{\mathfrak{gl}}}}
\nc{\hg}{{\widehat{\mathfrak{g}}}}
\nc{\hb}{{\widehat{\mathfrak{b}}}}
\nc{\hn}{{\widehat{\mathfrak{n}}}}

\nc{\fA}{{\mathfrak{A}}}
\nc{\fB}{{\mathfrak{B}}}
\nc{\fO}{{\mathfrak{O}}}
\nc{\fD}{{\mathfrak{D}}}
\nc{\fE}{{\mathfrak{E}}}
\nc{\fF}{{\mathfrak{F}}}
\nc{\fG}{{\mathfrak{G}}}
\nc{\fK}{{\mathfrak{K}}}
\nc{\fL}{{\mathfrak{L}}}
\nc{\fC}{{\mathfrak{C}}}
\nc{\fM}{{\mathfrak{M}}}
\nc{\fN}{{\mathfrak{N}}}
\nc{\fH}{{\mathfrak{H}}}
\nc{\fP}{{\mathfrak{P}}}
\nc{\fU}{{\mathfrak{U}}}
\nc{\fV}{{\mathfrak{V}}}
\nc{\fZ}{{\mathfrak{Z}}}
\nc{\fz}{{\mathfrak{z}}}

\nc{\bc}{{\mathbf{c}}}
\nc{\bd}{{\mathbf{d}}}
\nc{\bh}{{\mathbf{h}}}
\nc{\be}{{\mathbf{e}}}
\nc{\ba}{{\mathbf{a}}}
\nc{\bj}{{\mathbf{j}}}
\nc{\bn}{{\mathbf{n}}}
\nc{\bp}{{\mathbf{p}}}
\nc{\bg}{{\mathbf{g}}}
\nc{\bq}{{\mathbf{q}}}
\nc{\bs}{{\mathbf{s}}}
\nc{\bu}{{\mathbf{u}}}
\nc{\bv}{{\mathbf{v}}}
\nc{\bx}{{\mathbf{x}}}
\nc{\by}{{\mathbf{y}}}
\nc{\bb}{{\mathbf{b}}}
\nc{\bw}{{\mathbf{w}}}
\nc{\bA}{{\mathbf{A}}}
\nc{\bK}{{\mathbf{K}}}
\nc{\bB}{{\mathbf{B}}}
\nc{\bC}{{\mathbf{C}}}
\nc{\bD}{{\mathbf{D}}}
\nc{\bH}{{\mathbf{H}}}
\nc{\bM}{{\mathbf{M}}}
\nc{\bN}{{\mathbf{N}}}
\nc{\bV}{{\mathbf{V}}}
\nc{\bW}{{\mathbf{W}}}
\nc{\bL}{{\mathbf{L}}}
\nc{\bU}{{\mathbf{U}}}
\nc{\bX}{{\mathbf{X}}}
\nc{\bI}{{\mathbf{I}}}
\nc{\bZ}{{\mathbf{Z}}}
\nc{\bS}{{\mathbf{S}}}

\nc{\sA}{{\mathsf{A}}}
\nc{\sB}{{\mathsf{B}}}
\nc{\sC}{{\mathsf{C}}}
\nc{\sD}{{\mathsf{D}}}
\nc{\sF}{{\mathsf{F}}}
\nc{\sH}{{\mathsf{H}}}
\nc{\sE}{{\mathsf{E}}}
\nc{\sG}{{\mathsf{G}}}
\nc{\sK}{{\mathsf{K}}}
\nc{\sM}{{\mathsf{M}}}
\nc{\sO}{{\mathsf{O}}}
\nc{\sQ}{{\mathsf{Q}}}
\nc{\sP}{{\mathsf{P}}}
\nc{\sV}{{\mathsf{V}}}
\nc{\sZ}{{\mathsf{Z}}}
\nc{\sfp}{{\mathsf{p}}}
\nc{\sr}{{\mathsf{r}}}
\nc{\sg}{{\mathsf{g}}}
\nc{\sh}{{\mathsf{h}}}
\nc{\sk}{{\mathsf{k}}}
\nc{\ssf}{{\mathsf{f}}}
\nc{\sv}{{\mathsf{v}}}
\nc{\ssh}{{\mathsf{h}}}
\nc{\sse}{{\mathsf{e}}}
\nc{\sfb}{{\mathsf{b}}}
\nc{\sfc}{{\mathsf{c}}}
\nc{\sd}{{\mathsf{d}}}

\nc{\Av}{\on{Av}}
\nc{\act}{\on{act}}
\nc{\Hom}{\on{Hom}}
\nc{\Ext}{\on{Ext}}
\nc{\Tor}{\on{Tor}}
\nc{\End}{\on{End}}
\nc{\Lie}{\on{Lie}}
\nc{\Loc}{\on{Loc}}
\nc{\IC}{\on{IC}}
\nc{\Aut}{\on{Aut}}
\nc{\rk}{\on{rk}}
\nc{\Sh}{\on{Sh}}
\nc{\Perv}{\on{Perv}}
\nc{\pos}{{\on{pos}}}
\nc{\Conv}{\on{Conv}}
\nc{\Sph}{\on{Sph}}
\nc{\Sym}{\on{Sym}}
\nc{\Rep}{\on{Rep}}
\nc{\RepH}{{\mc R}ep(H)}
\nc{\Fun}{\on{Fun}}
\nc{\Id}{\on{Id}}
\nc{\id}{\on{id}}
\renc{\mod}{\on{--mod}}

\nc{\oG}{\overset{\circ}{G}{}}
\nc{\oGB}{{\overset{\circ}{G/B}{}}}
\nc{\oGN}{{\overset{\circ}{G/N}{}}}
\nc{\uBC}{\underline{\BC}}

\nc{\crit}{{\on{crit}}}
\nc{\reg}{{\on{reg}}}
\nc{\nilp}{{\on{nilp}}}
\nc{\ord}{\on{ord}}
\nc{\nil}{\wt{\on{reg}}}
\nc{\mb}{\mathbf}
\nc{\ren}{{\on{ren}}}
\nc{\res}{\on{res}}
\nc{\RS}{{\on{RS}}}
\nc{\Dist}{\on{Dist}}
\nc{\semiinf}{{\frac{\infty}{2}}}
\nc{\semiinfi}{{\frac{\infty}{2}+i}}
\nc{\semiinfb}{{\frac{\infty}{2}+\bullet}}
\nc{\torsemiinf}{{\overset{\semiinf}\otimes}}
\nc{\Hitch}{\on{Hitch}}

\nc{\hl}{\overset{\leftarrow}h}
\nc{\hr}{\overset{\rightarrow}h}
\nc\Dh{\widehat{\D}}
\nc{\Gr}{\on{Gr}}
\nc{\Grb}{\ol{\Gr}{}_G}
\nc{\Fl}{\on{Fl}}
\nc{\Flt}{\wt{\Fl}{}}
\nc{\Pic}{\on{Pic}}
\nc{\Bun}{\on{Bun}}

\nc{\bDR}{\mathbf {DR}}
\nc{\uV}{\underline{V}}
\nc{\arrowtimes}{\overset{\to}\otimes}
\nc{\hattimes}{\widehat\otimes}
\nc{\larrowtimes}{\overset{\leftarrow}\otimes}
\nc{\shriektimes}{\overset{!}\otimes}
\nc{\startimes}{\overset{*}\otimes}
\nc{\sCliff}{\mathsf {Cliff}}
\nc{\sSpin}{\mathsf {Spin}}

\nc{\one}{{\mathbf{1}}}

\nc\Spec{\on{Spec}}
\nc{\Pro}{\on{Pro}}
\nc{\QCoh}{\on{QCoh}}
\nc{\uHom}{\underline{\on{Hom}}}
\nc{\RHom}{\on{RHom}}
\nc{\uRHom}{\underline{\on{RHom}}}
\nc{\CHom}{{\mathcal Hom}}
\nc{\uCHom}{\underline{{\mathcal Hom}}}
\nc{\uCRHom}{\underline{{\mathcal R}{\mathcal Hom}}}

\nc{\cg}{{\check \fg}}
\nc{\Op}{\on{Op}_{\cG}}
\nc{\nOp}{\on{Op}^{\nilp}_{\cG}}
\nc{\nMOp}{\on{MOp}^{\nilp}_{\cG}}
\nc{\rOp}{\on{Op}^{\reg}_{\cG}}

\nc{\tg}{\wt{\check \fg}}
\nc{\cn}{\check \fn}
\nc{\tn}{\wt{\cn}}
\nc{\cG}{{\check G}}
\nc{\cB}{\check B}
\nc{\cT}{\check T}
\nc{\cH}{\check H}
\nc{\cb}{\check \fb}
\nc{\cN}{\check N}
\nc{\MOp}{\on{MOp}}
\nc{\tN}{\wt{\CN}_{\cG}}
\nc{\dIsom}{{\mathsf{Isom}}_{\Op}}
\nc{\disom}{{\mathsf{isom}}_{\Op}}
\nc{\Kdv}{{\mathsf{Isom}}_{\Op^\reg}}
\nc{\kdv}{{\mathsf{isom}}_{\Op^\reg}}
\nc{\Isom}{{\mathsf{Isom}}}
\nc{\isom}{{\mathsf{isom}}}

\nc{\wcosta}{j_{\wt{w},*}}
\nc{\wsta}{j_{\wt{w},!}}
\nc{\wcost}{j_{w,*}}
\nc{\wst}{j_{w,!}}

\nc{\epsi}{{\mathbf e}^\psi}
\nc{\epsip}{{\mathbf e}^{\psi'}}

\nc{\Ppi}{{\mathbf \Pi}}

\nc{\hCO}{{\hat{\CO}}}
\nc{\hCK}{{\hat{\CK}}}

\nc{\CPreg}{\CP_{G,\on{Op}^\reg}}
\nc{\CPBreg}{\CP_{B,\on{Op}^\reg}}
\nc{\CPnilp}{\CP_{G,\on{Op}^\nilp}}
\nc{\CPBnilp}{\CP_{B,\on{Op}^\nilp}}
\nc{\CPla}{\CP_{G,\on{Op}_{\cla}}}
\nc{\CPBla}{\CP_{B,\on{Op}_{\cla}}}

\nc{\Cat}{\hg_\crit\mod^{I,m}_\nilp}
\nc{\Catf}{{}^{fl}\hg_\crit\mod^{I,m}_\nilp}
\nc{\DCat}{D^b(\hg_\crit\mod_\nilp)^{I^0}}
\nc{\DCatf}{{}^{fl} D^b(\hg_\crit\mod_\nilp)^{I^0}}
\nc{\Catr}{\hg_\crit\mod^{I,m}_\reg}
\nc{\Catrf}{{}^{fl}\hg_\crit\mod^{I,m}_\reg}
\nc{\DCatr}{D^b(\hg_\crit\mod_\reg)^{I^0}}
\nc{\DCatrf}{{}^{fl} D^b(\hg_\crit\mod_\reg)^{I^0}}

\nc{\Z}{{\mathbb Z}}
\nc{\C}{{\mathbb C}}
\nc{\pone}{{\mathbb C}{\mathbb P}^1}
\nc{\pa}{\partial}
\nc{\F}{{\mathcal F}}
\nc{\arr}{\rightarrow}
\nc{\larr}{\longrightarrow}
\nc{\al}{\alpha}
\nc{\ri}{\rangle}
\nc{\lef}{\langle}
\nc{\W}{{\mathcal W}}
\nc{\la}{\lambda}
\nc{\ep}{\epsilon}
\nc{\su}{\widehat{{\mathfrak s}{\mathfrak l}}_2}
\nc{\sw}{{\mathfrak s}{\mathfrak l}}
\nc{\g}{{\mathfrak g}}
\nc{\h}{{\mathfrak h}}
\nc{\n}{{\mathfrak n}}
\nc{\N}{\widehat{\n}}
\nc{\De}{\Delta}
\nc{\gt}{\widetilde{\g}}
\nc{\Ga}{\Gamma}
\nc{\z}{{\mathfrak Z}}
\nc{\La}{\Lambda}
\nc{\cri}{_{\kappa_c}}
\nc{\kk}{h^\vee}
\nc{\sun}{\widehat{\sw}_N}
\nc{\si}{\sigma}
\nc{\el}{\ell}
\nc{\bi}{\bibitem}
\nc{\om}{\omega}
\nc{\ds}{\displaystyle}
\nc{\dzz}{\frac{dz}{z}}
\nc{\Res}{\on{Res}}
\nc{\Cal}{\mathcal}
\nc{\ot}{\otimes}
\nc{\R}{{\mc R}}
\nc{\yy}{{\mc Y}}
\nc{\ga}{\gamma}

\nc{\us}{\underset}
\nc{\opl}{\oplus}
\nc{\beq}{\begin{equation}}
\nc{\Fq}{{\mathbb F}_q}
\nc{\Mq}{{\mathcal M}}
\nc{\lan}{\langle}
\nc{\ran}{\rangle}

\nc{\Vect}{\on{Vect}}
\nc{\ghat}{\wh\fg}
\nc{\T}{\mc T}
\nc{\Tloc}{\T^\g_{\on{loc}}}
\nc{\vac}{|0\ran}
\nc{\Wick}{{\mb :}}
\nc{\delz}{\partial_z}
\nc{\K}{{\cali K}}
\nc{\cali}{\mathcal}
\nc{\li}{\mathfrak l}
\nc{\lt}{\widetilde{\li}}
\nc{\astar}{a^*}
\nc{\cA}{{\mc A}}
\nc{\ka}{\kappa}

\nc{\OO}{{\mc O}}
\nc{\AutO}{\on{Aut}\OO}
\nc{\DerO}{\on{Der}\OO}
\nc{\DerpO}{\on{Der}_+\OO}
\nc{\V}{{\mc V}}
\nc{\hh}{\wh{\h}}

\nc{\pp}{{\mathfrak p}}
\nc{\mm}{{\mathfrak m}}
\nc{\rr}{{\mathfrak r}}
\nc{\ket}{\rangle}
\nc{\zz}{{\mathfrak z}}
\nc{\gr}{\on{gr}}
\nc{\Spe}{\on{Spec}}
\nc{\rv}{\crho}
\nc{\can}{\on{can}}
\nc{\Db}{{\mathbb D}}
\nc{\ww}{w}

\nc{\RR}{\on{R}}
\nc{\PPi}{{\mathbf \Pi}}
\nc{\M}{{\mathbb M}}
\nc{\Mv}{{\mathbb M}^\vee}
\nc{\VV}{{\mathbb V}}
\nc{\bsl}{\backslash}

\nc{\bchi}{{\mathbf {\chi}}}
\nc{\anch}{{\mathbf {anch}}}

\nc{\cla}{{\check{\la}}}
\nc{\cmu}{{\check{\mu}}}
\nc{\crho}{{\check{\rho}}}
\nc{\com}{{\check{\omega}}}
\nc{\DD}{{\mc D}}
\nc{\E}{{\mc E}}
\nc{\Ll}{{\mc L}}

\nc{\ConnX}{\on{Conn}_{\cH}(\omega^{\rho}_X)}
\nc{\ConnD}{\on{Conn}_{\cH}(\omega_{\D}^{\rho})}
\nc{\ConnDt}{\on{Conn}_{\cH}(\omega_{\D^\times}^{\rho})}

\nc{\Hecke}{{\on{Hecke}}}

\nc{\cLambda}{{\check\Lambda}}
\nc{\cnu}{{\check\nu}}
\nc{\ceta}{{\check\eta}}

\nc{\Ind}{\on{Ind}}

\nc{\CTop}{{\mathcal Top}}

\nc{\ppart}{(\!(t)\!)}

\nc{\qu}{/\!/}

\nc{\gen}{\on{gen}}

\nc{\alg}{\on{alg}}
\nc{\geom}{\on{geom}}
\nc{\Jet}{\on{Jets}}
\nc{\aut}{\on{aut}}
\nc{\Der}{\on{Der}}

\nc{\uW}{\underline{W}}
\nc{\uU}{\underline{U}}

\nc{\aff}{\on{aff}}
\nc{\dirsum}{\oplus}

\nc{\Vl}{\BV_{\fg,\crit}^\la}

\nc{\gmod}{\hg_\crit\mod}
\nc{\ch}{{\on{ch}}}
\nc{\univ}{{\on{univ}}}

\begin{document}

\title{Local geometric Langlands correspondence: the spherical case}

\dedicatory{To Masaki Kashiwara on his 60th birthday}

\author[Edward Frenkel]{Edward Frenkel$^1$}\thanks{$^1$Supported by
DARPA and AFOSR through the grant FA9550-07-1-0543.}

\address{Department of Mathematics, University of California,
  Berkeley, CA 94720, USA}

\email{frenkel@math.berkeley.edu}

\author[Dennis Gaitsgory]{Dennis Gaitsgory$^2$}\thanks{$^2$Supported by
NSF grant 0600903.}

\address{Department of Mathematics, Harvard University,
Cambridge, MA 02138, USA}

\email{gaitsgde@math.harvard.edu}

\date{November 2007}

\begin{abstract}
A module over an affine Kac--Moody algebra $\ghat$ is called spherical
if the action of the Lie subalgebra $\g[[t]]$ on it integrates to an
algebraic action of the corresponding group $G[[t]]$. Consider the
category of spherical $\ghat$-modules of critical level. In this paper
we prove that this category is equivalent to the category of
quasi-coherent sheaves on the ind-scheme of opers on the punctured
disc which are unramified as local systems. This result is a
categorical version of the well-known description of spherical vectors
in representations of groups over local non-archimedian fields. It may
be viewed as a special case of the local geometric Langlands
correspondence proposed in \cite{FG2}.

\end{abstract}

\maketitle

\section{Introduction}

A general framework for the local geometric Langlands correspondence was
proposed in our earlier work \cite{FG2} (see also
\cite{FG3}--\cite{FG5} and \cite{F:book}). According to our proposal,
to each ``local Langlands parameter'' $\sigma$, which is a
$\cG$--local system on the punctured disc $\D^\times = \on{Spec}
\C\ppart$ (or equivalently, a $\cG$-bundle with a connection on
$\D^\times$), there should correspond a category $\CC_\sigma$ equipped
with an action of the formal loop group $G\ppart$. Even more
ambitiously, we expect that there exists a category ${\mc C}_{\univ}$ fibered
over the stack $\Loc_{\cG}(\D^\times)$ of $\cG$-local systems, equipped
with a fiberwise action of the ind-group $G\ppart$, whose fiber
category at $\sigma$ is ${\mc C}_\sigma$. Moreover, we expect 
${\mc C}_{\univ}$ to be the universal category equipped with an action of
$G\ppart$. In other words, we expect that $\Loc_{\cG}(\D^\times)$ is
the universal parameter space for the categorical representations of
$G\ppart$. The ultimate form of the local Langlands correspondence
for loop groups should be, roughly, the following statement:

\medskip

\begin{equation}    \label{ultimate}
\boxed{\begin{matrix} \text{categories fibering} \\
\text{over } \Loc_{\cG}(\D^\times) \end{matrix}} \quad
\Longleftrightarrow  \quad \boxed{\begin{matrix} \text{categories
equipped} \\ \text{with action of } G\ppart \end{matrix}}
\end{equation}

\bigskip

We should point out, however, that neither the notion of category
fibered over a non-algebraic stack such as $\Loc_{\cG}(\D^\times)$,
nor the unversal property alluded to above are easy to formulate. So
for now \eqref{ultimate} should be understood heuristically, as a
guiding principle.

\medskip

As we explained in \cite{FG2}, the local geometric Langlands
correspondence should be viewed as a categorification of the local
Langlands correspondence for the group $G(F)$, where $F$ is a local
non-archimedian field. This means that the categories ${\mc
C}_\sigma$, equipped with an action of $G\ppart$, that we wish to
attach to the Langlands parameters $\sigma \in \Loc_{\cG}(\D^\times)$
should be viewed as categorifications of smooth representations of
$G(F)$ in the sense that we expect the Grothendieck groups of the
categories $\CC_\sigma$ to ``look like'' irreducible smooth
representations of $G(F)$.

\subsection{The spherical part}

In the study of representations $\pi$ of $G(F)$, a standard tool is
to consider the subspaces $\pi^K$ of vectors fixed by open compact
subgroups $K$ of $G(F)$.

\medskip

This procedure has a categorical counterpart. Let $K$ be a
group-scheme contained in $G[[t]]$ and containing the $N$th congruence
subgroup $K_N$ for some $N$ (i.e., the subgroup of $G[[t]]$ consisting
of elements congruent to $1$ modulo $t^N\C[[t]]$). For example, $K$
can be $G[[t]]$ itself, or the Iwahori subgroup $I$.

\medskip

Given a category $\CC$, acted on by $G\ppart$, we can consider the
corresponding $K$-equivariant category $\CC^K$. Via \eqref{ultimate},
any such $\CC^K$ is also a category fibered over
$\Loc_{\cG}(\D^\times)$.

\medskip

This procedure applies in particular to ${\mc C}_{\univ}$. Although at
present, we do not know how to construct the entire category 
${\mc C}_{\univ}$, we do have a guess what ${\mc C}_{\univ}^K$ for
some choices of $K$.

\medskip

In this paper we specialize to the simplest case $K=G[[t]]$. (Another
case, which can be explicitly analyzed is that of $K=I$, discussed in
\cite{FG2}.) Based on the analogy with the classical local Langlands
correspondence for spherical representations, we propose:

\begin{equation} \label{spherical universal}
\CC_{\univ}^{G[[t]]}\simeq \Rep(\cG).
\end{equation}

Here $\Rep(\cG)$ is the category of (algebraic) representations of
$\cG$, which can be also thought as the category of quasi-coherent
sheaves on the stack $\on{pt}\hspace*{-1mm}/\cG$. The structure of
category fibered over $\Loc_{\cG}(\D^\times)$ comes from the maps of
stacks
\begin{equation}    \label{unr}
\on{pt}\hspace*{-1mm}/\cG\simeq \Loc_{\cG}^{\on{unr}}\to
\Loc_{\cG}(\D^\times)
\end{equation}
corresponding to the inclusion of the stack $\Loc_{\cG}^{\on{unr}}$ of
unramified local systems (or, equivalently, local systems on the
unpunctured disc $\D$) into the stack $\Loc_{\cG}(\D^\times)$ of all
local systems.

\subsection{Representations of critical level}

In \cite{FG2} we have considered a specific example of a category
equipped with an action of $G\ppart$; namely, the category
$\hg_\crit\mod$ of modules over the affine Kac--Moody algebra $\hg$ of
critical level. It carries a canonical action
of the ind-group $G\ppart$ via its adjoint action on $\hg_\crit$. 

\medskip

What should be the relationship between $\gmod$ and
the conjectural universal category ${\mc C}_{\univ}$? 

\medskip

We note that the category $\gmod$ naturally fibers over the ind-scheme
$\Op(\D^\times)$ of $\cG$-opers on $\D^\times$ introduced in
\cite{BD}.  This is because, according to \cite{FF,F:wak}, the center
$\fZ_{\fg}$ of the category $\gmod$ is isomorphic to the algebra of
functions on $\Op(\D^\times)$.

\medskip

The idea of
\cite{FG2} is that the latter fibration is a ``base change'' of
$\CC_{\univ}$, that is, there is a Cartesian diagram
\begin{equation}    \label{base change}
\begin{CD}
\gmod @>>> {\mc C}_{\univ} \\
@VVV @VVV \\
\Op(\D^\times) @>{\al}>> \Loc_{\cG}(\D^\times)
\end{CD}
\end{equation}
which commutes with the action of $G\ppart$ along the fibers of the two
vertical maps. In other words,
\begin{equation}    \label{main eq}
\gmod \simeq {\mc C}_{\univ} \underset{\Loc_{\cG}(\D^\times)}\times
\Op(\D^\times).
\end{equation}

\medskip

Given a $\cg$-oper $\chi$, let us consider it as a point of $\Spec(\fZ_\fg)$,
i.e., a character of $\fZ_\fg$. Let $\hg_\crit\mod_\chi$ be the full
subcategory $\hg_\crit\mod_\chi$ of $\hg_\crit\mod$ whose
objects are $\ghat_\crit$-modules, on which the $\fZ_\fg$ acts
according to this character. This is the fiber category of the
category $\gmod$ over $\chi \in \Op(\D^\times)$. 

\medskip

Let $\sigma=\al(\chi)\in \Loc_{\cG}(\D^\times)$. By
\eqref{main eq}, we have:
\begin{equation} \label{oper eq}
\CC_\sigma\simeq \hg_\crit\mod_\chi.
\end{equation} 

\medskip

As was mentioned above, at the moment 
we do not have an independent definition of ${\mc C}_{\univ}$, and
therefore we cannot make the equivalences \eqref{main eq} 
and \eqref{oper eq} precise. But
we use it as our guiding principle. This leads us to a number of
interesting corollaries, some of which have been discussed in
\cite{FG2}--\cite{FG5}.

\medskip

For example, if $\chi, \chi'$ are two $\cG$-opers, such that the
corresponding local systems $\al(\chi)$ and $\al(\chi')$ are isomorphic,
for every choice of an isomorphism we are supposed to have 
an equivalence of categories:
\begin{equation} \label{indep of oper}
\gmod_\chi\simeq \gmod_{\chi'}.
\end{equation}
This is a highly non-trivial conjecture about
representations of $\hg_\crit$.

\subsection{Harish-Chandra categories}

Let us return to the discussion of the category of $K$-equivariant
objects in the context of $\CC=\gmod$. The corresponding category
$\gmod^K$ identifies with the category of $(\ghat_{\crit},K)$
Harish-Chandra modules. When $K$ is connected, this is a full abelian
subcategory of $\gmod$, consisting of modules, on which the action of
the Lie algebra $\on{Lie}(K) \subset \ghat_{\crit}$ is integrable,
i.e., comes from an algebraic action of $K$.

\medskip

Now specialize to the case $K=G[[t]]$. We call objects of the
corresponding category $\gmod^{G[[t]]}$ of $G[[t]]$-equivariant
$\ghat_\crit$-modules {\em spherical}. Combining
eqns. \eqref{spherical universal}, \eqref{unr} and \eqref{main eq}, we
arrive at the following equivalence:

\begin{equation} \label{spherical HCH}
\gmod^{G[[t]]}\simeq \QCoh\left(\Loc^{\on{unr}}_{\cG}
\underset{\Loc_{\cG}(\D^\times)}\times \Op(\D^\times)\right).
\end{equation}

Here we should remark that although the stack $\Loc_{\cG}(\D^\times)$ is
a problematic object to work with, the fiber product 
$$\Loc^{\on{unr}}_{\cG} \underset{\Loc_{\cG}(\D^\times)}\times
\Op(\D^\times)$$ appearing on the right-hand side of \eqref{spherical
HCH} is a well-defined (non-reduced) ind-subscheme of
$\Op(\D^\times)$. This is the moduli ind-scheme of opers that are
unramified as local systems. We denote this ind-scheme by
$\Op^{\on{unr}}$. It is a disjoint union of formal schemes
$\Op^{\on{unr},\la}$, $\la$ being a dominant weight, where the
reduced scheme corresponding to each $\Op^{\on{unr},\la}$ is the
scheme $\Op^{\reg,\la}$ of $\la$-regular opers introduced in
\cite{FG2}.

\medskip

Thus, the heuristic guess given by \eqref{spherical HCH} leads to the
following precise statement, which is the main result of this paper:

\medskip

\noindent{\bf Main Theorem.} {\em The category $\gmod^{G[[t]]}$ of
spherical $\ghat_\crit$-modules is equivalent to the category of
quasi-coherent sheaves on the ind-scheme $\Op^{\on{unr}}$ of
$\cG$-opers on $\D^\times$ unramified as local systems.}

\medskip

Moreover, we show that a functor from the former category to the
latter one is an analogue of the Whittaker functor.

\subsection{Some corollaries}

Let $\chi$ be a $\BC$-point of $\Op(\D^\times)$, and let us consider
the category $\gmod^{G[[t]]}_\chi$. As an abelian category, this is a
full subcategory of $\gmod_\chi$, consisting of $G[[t]]$-integrable
modules.

\medskip

One can show (see \cite{FG3}, Corollary 1.11) that this category is
$0$ unless $\chi\in \Op^{\on{unr}}$. In the latter case, from the Main
Theorem we obtain that the category $\gmod^{G[[t]]}_\chi$ is
equivalent to the category of vector spaces. This result is the first
test for our prediction that $\gmod_\chi$, as a category equipped with
a $G\ppart$-action, depends only on $\alpha(\chi)$, as expected in
\eqref{indep of oper}. In addition, this equivalence is in agreement
with a classical fact that the space of spherical vectors in an
irreducible representation of $G(F)$ is either zero or
one-dimensional.

\medskip

As another corollary of the Main Theorem, we obtain the following
description of the algebra of self-Exts of the Weyl modules $\BV_\la$
in the derived category $D(\gmod^{G[[t]]})$ of
$(\ghat_\crit,G[[t]])$ Harish-Chandra modules:
$$
\Ext^\bullet_{D(\hg_\crit\mod^{G[[t]]})}(\Vl,\Vl)\simeq
\Lambda^\bullet_{\fz^{\reg,\la}_\fg} (\CN^\la_{\reg/\on{unr}}),
$$
where $\CN^\la_{\reg/\on{unr}}$ is the bundle of $\Op^{\reg,\la}$ in
$\Op^{\on{unr},\la}$. (In the above formula we identify the algebra
of function on $\Op^{\reg,\la}$ with the corresponding quotient of
$\fZ_\fg$, denoted $\fz^{\reg,\la}_\fg$.) For $\la=0$ this isomorphism
was previously established in \cite{FT} by other methods.

\subsection{Structure of the proof}

The proof of the Main Theorem is quite simple. The main idea is that
the category $\gmod^{G[[t]]}$ has a universal object, denoted
$\fD^{\on{ch}}_{G,\crit}$, which is the vacuum module of the chiral
algebra of differential operators on $G$.  The module
$\fD^{\on{ch}}_{G,\crit}$ is in fact a $\hg_\crit$-bimodule, and for
any other object $\CM\in \gmod^{G[[t]]}$ we have
$$\CM\simeq
\fD^{\on{ch}}_{G,\crit}\overset{\frac{\infty}{2}}{\underset{\fg\ppart}
\otimes} \CM$$
(here $\overset{\frac{\infty}{2}}{\underset{\fg\ppart} \otimes}$
stands for the semi-infinite Tor functor).

Therefore, in order to define functors and check isomorphisms on
$\gmod^{G[[t]]}$, it is enough to do so just for the module
$\fD^{\on{ch}}_{G,\crit}$. Thus, in Sect. 2 we prove a theorem that
describes the structure of $\fD^{\on{ch}}_{G,\crit}$ as a bi-module
over $\gmod^{G[[t]]}$, and in Sect. 3 we derive our Main Theorem from
this structure theorem.

\section{Chiral differential operators on $G$ at the critical level}

In this section we describe the structure of the chiral algebra of
differential operators (CADO) on a simple connected simply-connected
algebraic group $G$ over $\C$ at the critical level, viewed as a
bimodule over $\ghat_\crit$.

\ssec{Notation}

We will follow the notation of \cite{FG2}. In particular,
$\ghat_\crit$ is the critical central extension of the formal loop
algebra $\g\ppart$, $\hg_\crit\mod$ is the category of discrete
modules over $\hg_\crit$, $\fZ_\fg$ is the center of $\hg_\crit\mod$
(or, equivalently, of the completion of the enveloping algebra of
$\ghat_\crit$). This is a topological commutative algebra. According
to a theorem of \cite{FF,F:wak}, the corresponding ind-scheme
$\Spec(\fZ_\fg)$ is canonically isomorphic to the moduli space
$\Op(\D^\times)$ of $\cG$-opers on the formal punctured disc, where
$\cG$ is the Langlands dual group to $G$ (of adjoint type). For the
definition of $\Op(\D^\times)$, see \cite{BD}.

\medskip

For $\la\in \Lambda^+$, we let $\fz^{\reg,\la}_\fg$ denote the
quotient of $\fZ_\fg$ corresponding to the sub-scheme
$\Op^{\reg,\la}\subset \Op(\D^\times)$ introduced in \cite{FG2},
Section 2.9. Let
$$
\BV_{\fg,\crit}^\la = \Ind^{\hg_\crit}_{\fg[[t]] \oplus \C {\bf
1}}(V^\la) :=U(\hg_\crit) \underset{U(\fg[[t]] \oplus \C {\mb 1})}
\otimes V^\la
$$
be the Weyl module with dominant integral highest weight $\la \in
\Lambda^+$. According to \cite{FG6}, Theorem 1, the action of $\fZ_\fg$ on
$\BV_{\fg,\crit}^\la$ factors as follows:
$$
\fZ_\fg\twoheadrightarrow \fz^{\reg,\la}_\fg\simeq \on{End}(\Vl).
$$
Furthermore, $\Vl$ is flat (and in fact, free) as a
$\fz^{\reg,\la}_\fg$-module.

\medskip

Let $\hg_\crit\mod^{G[[t]]}$ be the full abelian subcategory of
$\hg_\crit\mod$. In this paper we will work with the "naive"
derived category $D(\hg_\crit\mod^{G[[t]]})$. However,
by generalizing the argument of \cite{FG2}, Sect. 20.16, 
one can identify $D(\hg_\crit\mod^{G[[t]]})$ with the 
$G[[t]]$-equivariant derived category corresponding
to $\hg_\crit\mod$, as introduced in {\it loc. cit.}, Sect. 20.8. 

\medskip

In particular, for $\CM\in \hg_\crit\mod^{G[[t]]}$ we have:
$$\Ext^i_{\hg_\crit\mod^{G[[t]]}}(\BV_{\fg,\crit}^\la,\CM)\simeq
\Ext^i_{G[[t]]}(V^\la,\CM).$$

\ssec{Unramified opers}

Let $\Op^{\on{unr}}\subset \Op(\D^\times)$ be the ind-subscheme of
opers that are unramified as local systems. For any $\C$-algebra $A$, the set of 
$A$-points of $\Op^{\on{unr}}$ is by definition the set of opers on $\on{Spec}
A\ppart$, which are isomorphic, as local systems, to the trivial local
system. We have:
\begin{equation}    \label{decomp}
\Op^{\on{unr}}\simeq \bigcup_{\la \in \Lambda^+} \,
\Op^{\on{unr},\la},
\end{equation}
where $\Op^{\on{unr},\la}$ are pairwise disjoint formal 
sub-schemes of $\Op(\D^\times)$ with
$$(\Op^{\on{unr},\la})_{\on{red}}\simeq \Op^{\reg,\la}.$$

We will also use the notation $\Spec(\fZ_\fg^{\on{unr}})$,
$\Spec(\fZ_\fg^{\on{unr},\la})$ for these ind-schemes. Let
$\iota^\la_{\reg/\on{unr}}$ denote the closed embedding
$\Spec(\fz_\fg^{\reg,\la})\hookrightarrow
\Spec(\fZ_\fg^{\on{unr},\la})$; let $\bI^\la$ denote the (closed)
ideal of $\Spec(\fz_\fg^{\reg,\la})$ in
$\Spec(\fZ_\fg^{\on{unr},\la})$; let $\CN^{\la}_{\reg/\on{unr}}$
be the normal scheme to $\Spec(\fz_\fg^{\reg,\la})$ in
$\Spec(\fZ_\fg^{\on{unr},\la})$. It follows from \cite{FG2}, Section
4.6, that its sheaf of sections is a locally free
$\fz_\fg^{\reg,\la}$-module (in other words,
$\CN^{\la}_{\reg/\on{unr}}$ is s vector bundle over
$\Spec(\fz_\fg^{\reg,\la})$).  Moreover, $\fZ_\fg$ carries a Poisson
structure, which identifies $\CN^{\la}_{\reg/\on{unr}}$ with
$\Omega^1(\fz_\fg^{\reg,\la})$.

\medskip

The following fact was established in \cite{FG3}, Corollary 1.11 (note
that $\Spec(\fZ_\fg^{\on{unr}})$ was denoted by
$\Spec(\fZ_\fg^{\on{m.f.}})$ in \cite{FG3}).

\begin{thm}  \label{support}
The support in $\Spec(\fZ_\fg)$ of every $\CM\in
\hg_\crit\mod^{G[[t]]}$ is contained in $\Spec(\fZ_\fg^{\on{unr}})$.
\end{thm}

Thus, every $G[[t]]$-integrable $\hg$-module $\CM$ splits as a direct
sum $\underset{\la}\bigoplus\, \CM^\la$, where $\CM^\la$ is supported
at $\Spec(\fZ_\fg^{\on{unr},\la})$ and has an increasing filtration 
whose sub-quotients are quotient modules of $\Vl$.

\ssec{Chiral differential operators}

Let $X$ be a smooth algebraic curve. We will work with Lie-* algebras
and chiral algebras on $X$ and with modules over them supported at a
fixed point $x \in X$ (see \cite{CHA} for the definitions). In this
paper all chiral algebras will come from vertex algebras, and we will
tacitly identify a chiral algebra with its vacuum module, i.e., its
fiber at any point of a curve equipped with a coordinate. In fact,
everything may be rephrased in terms of the corresponding vertex (Lie)
algebras and modules over them, but we will use the formalism of
chiral algebras for the sake of consistency with
\cite{FG1}--\cite{FG3}.

Recall from \cite{AG,GMS:hom} that for any level $\kappa$ (i.e., an
invariant bilinear form on $\g$) we have the chiral algebra of
differential operators (CADO), denoted by $\fD_{G,\kappa}^{\ch}$. It
comes equipped with two mutually commuting embeddings
\begin{equation}    \label{two actions}
\CA_{\fg,\kappa} \overset{{\mathfrak l}_{\fg}}\longrightarrow
\fD^{\on{\ch}}(G)_\kappa \overset{{\mathfrak r}_{\fg}}\longleftarrow
\CA_{\fg,\kappa'},
\end{equation}
where
$$
\kappa'=-\kappa+2\kappa_\crit.
$$

Recall that the fiber of $\CA_{\fg,\kappa}$ at $x$ is the vacuum Weyl
module $\BV_{\fg,\kappa}$ of level $\kappa$, and the fiber of
$\fD_{G,\crit}^{\ch}$ at $x$ with
$$
\Ind^{\hg_\kappa}_{\fg[[t]] \oplus \C {\mb
1}}(\CO_{G[[t]]}):=U(\hg_\kappa) \underset{U(\fg[[t]] \oplus \C {\mb
1})}\otimes \CO_{G[[t]]}.
$$
Here $t$ is a formal coordinate at $x$ and $\CO_{G[[t]]}$ is the
algebra of functions on the group $G[[t]]$, on which $\fg[[t]]$ acts
trivially and ${\mb 1}$ acts as the identity. This is a module over
$\ghat_\ka \oplus \ghat_{\ka'}$. The action of $\ghat_\ka$ on
$\fD^{\on{\ch}}(G)_{\kappa,x}$ corresponding to the left arrow in
\eqref{two actions} is the natural action on this induced module, and
the action corresponding to the right arrow in \eqref{two actions} was
constructed in \cite{AG,GMS:hom}. We will refer to the two actions as
the ``left'' and the ``right'' actions, respectively.

\ssec{CADO at the critical level}

We now specialize to the critical level $\ka=\ka_\crit$. Then $\kappa'
= \kappa_\crit$, and so both left and right actions of $\ghat$
correspond to the critical level. We will describe the structure of
$\fD^{\on{\ch}}(G)_{\crit,x}$ as a $\ghat_\crit$-bimodule.
From now on, when there is no confusion, we will skip the subscript 
$x$ when describing the fiber of the chiral algebra at $x$.

\medskip

Let $\fz_\fg$ denote the center of $\CA_{\fg,\crit}$ (note that 
$\fz_\fg$ identifies with $\fz_\fg^{\reg,0}$). The following has
been established in \cite{FG1}:

\begin{lem}  \label{two embeddings}
The two embeddings
$$\fl,\fr:\fz_\fg \rightrightarrows \fD_{G,\crit}^{\ch}$$
differ by the automorphism of $\fz_\fg$, induced by 
Cartan involution $\tau$ of $\fg$.
\end{lem}

\medskip

As a bimodule over $\hg_\crit$, $\fD_{G,\crit}^{\ch}$ is
$G[[t]]$-integrable with respect to both actions. By \thmref{support},
its support over $\Spec(\fZ_\fg)$ is contained in
$\Spec(\fZ_\fg^{\on{unr}})$ (note that by \lemref{two embeddings},
the two actions of $\fZ_\fg$ on $\fD_{G,\crit}^{\ch}$ differ by
$\tau$). Hence, we have a direct sum decomposition of
$\fD_{G,\crit}^{\ch}$ as a $\hg_\crit$-bimodule:
$$\fD_{G,\crit}^{\ch}\simeq \bigoplus_{\la \in \Lambda^+}\,
\fD_{G,\crit}^{\ch,\la}$$
where $\fD_{G,\crit}^{\ch,\la}$ is the summand supported at
$\Op^{\on{unr},\la} = \Spec(\fZ_\fg^{\on{unr},\la})$ (see formula
\eqref{decomp}).

\medskip

Recall that $\CO_{G[[t]]}$ denotes the algebra of functions on
$G[[t]]$. It has a natural structure of commutative chiral algebra,
and as such it is a chiral subalgebra of $\fD_{G,\crit}^{\ch}$. The
map
$$\CO_{G[[t]]}\to \fD_{G,\crit}^{\ch}$$
respects the bimodule structure with respect to $\fg[[t]]
\subset \ghat_\crit$.

\medskip

For $\la\in \Lambda^+$ we have a natural map
$$V^\la\otimes V^{\tau(\la)}\to \CO_G\hookrightarrow \CO_{G[[t]]},$$
compatible with the action of $\fg[[t]]\oplus \fg[[t]]$. Inducing, we
obtain a map of bimodules over $\hg_\crit$:
$$\Vl\otimes \BV_{\fg,\crit}^{\tau(\la)}\to \fD_{G,\crit}^{\ch,\la}.$$

{}From \lemref{two embeddings} we obtain:

\begin{lem}  
The above map factors through a map
\begin{equation} \label{zero-th term}
\Vl\underset{\fz^{\reg,\la}_\fg}\otimes \BV_{\fg,\crit}^{\tau(\la)}\to 
\fD_{G,\crit}^{\ch,\la}.
\end{equation}
\end{lem}

\ssec{A description of the CADO}

Recall that $\bI^\la$ denotes the ideal of $\Spec(\fz_\fg^{\reg,\la})$
in $\Spec(\fZ_\fg^{\on{unr},\la})$. Consider the canonical
increasing filtration on $\fD_{G,\crit}^{\ch,\la}$ numbered by
$i=0,1,...$ with $F^i(\fD_{G,\crit}^{\ch,\la})$ being the $\hg_\crit$
sub-bimodule, annihilated by the $i+1$-st power of the ideal
$\bI^\la$. By construction, the image of the map \eqref{zero-th term}
belongs to $F^0(\fD_{G,\crit}^{\ch,\la})$.

\medskip

We are now ready to formulate the main result of this section:

\begin{thm}  \label{structure}  \hfill

\smallskip

\noindent{\em (1)}
The map \eqref{zero-th term} defines an isomorphism
$$\Vl\underset{\fz^{\reg,\la}_\fg}\otimes
\BV_{\fg,\crit}^{\tau(\la)}\simeq F^0(\fD_{G,\crit}^{\ch,\la}).$$

\smallskip

\noindent{\em (2)}
The canonical maps
$$(\bI^\la)^n/(\bI^\la)^{n+1}\underset{\fz^{\reg,\la}_\fg}\otimes
\on{gr}^n(\fD_{G,\crit}^{\ch,\la})\to
\on{gr}^0(\fD_{G,\crit}^{\ch,\la})$$ give rise to isomorphisms
\begin{equation}    \label{consecutive}
\on{gr}^n(\fD_{G,\crit}^{\ch,\la})\simeq
\on{gr}^0(\fD_{G,\crit}^{\ch,\la}) \underset{\fz^{\reg,\la}_\fg}\otimes
\on{Sym}^n_{\fz^{\reg,\la}_\fg}(\CN^\la_{\reg/\on{unr}})
\end{equation}
of $\ghat_\crit$-bimodules, where $\CN^\la_{\reg/\on{unr}}$ is the
normal bundle to $\Op^{\reg,\la}$ in $\Op^{\on{unr},\la}$.
\end{thm}

The above theorem should be contrasted with the following:

\begin{lem}
For a generic $\ka$ (i.e., such that $\ka/\ka_c$ is not a rational
number) we have an isomorphism
\begin{equation}    \label{generic}
\fD^{\on{\ch}}_{G,\ka} \simeq \bigoplus_{\la \in \Lambda^+} \BV_{\fg,\ka}^\la
\otimes \BV_{\fg,\ka'}^{\tau(\la)},
\end{equation}
of $\ghat_\ka \oplus \ghat_{\ka'}$ modules.
\end{lem}

\begin{proof}
For any level $\ka$ we have a canonical non-zero homomorphism of
$\ghat_\ka \oplus \ghat_{\ka'}$ modules
\begin{equation}    \label{any ka}
\BV_{\fg,\ka}^\la \otimes \BV_{\fg,\ka'}^{\tau(\la)} \to
\fD^{\on{\ch}}_{G,\ka}.
\end{equation}
If $\ka$ satisfies the conditions of the lemma, then both
$\BV_{\fg,\ka}^\la$ and $\BV_{\fg,\ka'}^{\tau(\la)}$ are irreducible
modules. Therefore the above maps are injective. The assertion of the
lemma then follows from the obvious fact that the characters of the
two sides of \eqref{generic} are equal to each other.
\end{proof}

For special values of $\ka$, when $\ka/\ka_c \in {\mathbb Q}$, the
modules $\BV_{\fg,\ka}^\la$ and $\BV_{\fg,\ka'}^{\tau(\la)}$ may become
reducible, and so the structure of $\fD^{\on{\ch}}_{G,\ka}$ may
become more complicated. \thmref{structure} describes what happens at
the critical level $\ka=\ka_\crit$. In this case the image of the
homomorphism \eqref{any ka} is equal to $\BV_{\fg,\crit}^\la
\underset{\zz_\la}\otimes \BV_{\fg,\crit}^{\tau(\la)}$. In other
words, we observe the collapse of the degrees of freedom corresponding
to $\zz_\la$. But these degrees of freedom are restored by the second
factor in \eqref{consecutive}.

\ssec{Proof of part (2)}

We shall first prove part (2) of \thmref{structure}. Recall the
chiral algebroid $\CA^{\ren,\tau}_{\fg,\crit}$ of \cite{FG1}, Section
4, whose chiral enveloping algebra is $\fD_{G,\crit}^{\ch,0}$, by Lemma
9.7 of {\it loc. cit.} (note that the assertion of \thmref{structure}
for $\la=0$ follows in fact from this isomorphism and Lemma 7.4 of
\cite{FG1}).

\medskip

A version of Kashiwara's theorem
proved in Section 7 of {\it loc. cit.} implies the following:

\begin{prop}  \label{Kashiwara}
Let $\CM$ be a chiral $\CA^{\ren,\tau}_{\fg,\crit}$-module, whose support
over $\Spec(\fZ_\fg)$ is contained in $\Spec(\fZ_\fg^{\on{unr},\la})$.
Let $F^i(\CM)$, $i=1,2,...$ be the canonical increasing filtration on
$\CM$ by the powers of $\bI^\la$. Then

\smallskip

\noindent{\em(a)}
$R^i(\iota^\la)^!(\CM)=0$ for $i>0$ and $R^0(\iota^\la)^!(\CM)\simeq
F^0(\CM)$.

\smallskip

\noindent{\em(b)} The canonical maps
$$(\bI^\la)^n/(\bI^\la)^{n+1}\underset{\fz^{\reg,\la}_\fg}\otimes
\on{gr}^n(\CM)\to \on{gr}^0(\CM)$$ give rise to isomorphisms
$$\on{gr}^n(\CM)\simeq \on{gr}^0(\CM)
\underset{\fz^{\reg,\la}_\fg}\otimes 
\on{Sym}^n_{\fz^{\reg,\la}_\fg}(\CN^\la_{\reg/\on{unr}}).$$
\end{prop}

We apply this proposition to $\fD_{G,\crit}^{\ch,\la}$, which is a
chiral module over $\fD_{G,\crit}^{\ch,0}$, and hence over
$\CA^{\ren,\tau}_{\fg,\crit}$, and the assertion of point (2) of
\thmref{structure} follows.

\ssec{Proof of part (1)}

To prove part (1) of \thmref{structure}, let us first show that the
map \eqref{zero-th term} is injective. Indeed, let $K$ denote its
kernel; this is a bimodule over $\hg_\crit$, supported at
$\Spec(\fZ_\fg^{\on{unr},\la})$ and $G[[t]]$-integrable with respect
to both actions. Hence, if $K\neq 0$, there exists a non-zero map of
$\Vl\to K$ of $\hg_\crit$-modules, with respect to the left action.

\medskip

However, we claim that the map
\begin{equation} \label{zero-th term, inv}
\Hom_{\hg_\crit} (\Vl,\Vl\underset{\fz^{\reg,\la}_\fg}\otimes
\BV_{\fg,\crit}^{\tau(\la)})\to
\Hom_{\hg_\crit}(\Vl,\fD_{G,\crit}^{\ch,\la})
\end{equation}
is injective, and in fact an isomorphism. This would lead to
a contradiction, implying that $K=0$.

\medskip

To show that \eqref{zero-th term, inv} is an isomorphism, consider the
composition
\begin{equation} \label{zero-th term, inv, comp}
\BV_{\fg,\crit}^{\tau(\la)}\to
\Hom_{\hg_\crit}(\Vl,\Vl\underset{\fz^{\reg,\la}_\fg}\otimes
\BV_{\fg,\crit}^{\tau(\la)})\to
\Hom_{\hg_\crit}(\Vl,\fD_{G,\crit}^{\ch,\la}).
\end{equation}

We claim that the first arrow in \eqref{zero-th term, inv, comp}
is an isomorphism. Indeed, since  
$\BV_{\fg,\crit}^{\tau(\la)}$ is flat over $\fz^{\reg,\la}_\fg$, we have
$$\Hom_{\hg_\crit}
(\Vl,\Vl\underset{\fz^{\reg,\la}_\fg}\otimes \BV_{\fg,\crit}^{\tau(\la)})\simeq
\End_{\hg_\crit}(\Vl)\underset{\fz^{\reg,\la}_\fg}\otimes
\BV_{\fg,\crit}^{\tau(\la)},$$ and by the main result of \cite{FG6},
the natural map
$$\fz^{\reg,\la}_\fg\to \End_{\hg_\crit}(\Vl)$$
is an isomorphism.

\medskip

Now we claim that the composition in \eqref{zero-th term, inv, comp}
is an isomorphism. The latter is equivalent to point (a) of the
following assertion, established in \cite{AG}:

\begin{lem}  \label{inv in Weyl} \hfill

\smallskip

\noindent{\em(a)}
$$\Hom_{\hg_\crit}(\Vl,\fD_{G,\crit}^{\ch})\simeq 
\Hom_{G[[t]]}(V^\la,\fD_{G,\crit}^{\ch})\simeq
\BV_{\fg,\crit}^{\tau(\la)}.$$

\smallskip

\noindent{\em(b)} For $i>0$,
$$\Ext^i_{G[[t]]}(V^\la,\fD_{G,\crit}^{\ch})=0.$$
\end{lem}

\ssec{Computation of characters}

We are now ready to finish the proof of \thmref{structure}. Using the
coordinate on the formal disc, we will view $\fD_{G,\crit}^{\ch}$ as acted
on by $G\times G\times \BG_m$, where the latter acts by loop rotations.
It is easy to see that the isotypic components for the above action
are finite-dimensional.

\medskip

Using point (2) of \thmref{structure} and the above injectivity
result, we obtain that the theorem would follow once we show that for
each $\mu_1,\mu_2,d$,
\begin{align*}
&\dim\left(\Hom_{G\times G\times \BG_m}(V^{\mu_1}\otimes V^{\mu_2}
\otimes \BC^d, \fD_{G,\crit}^{\ch})\right)= \\ &\underset{\la}\sum\,
\dim\left(\Hom_{G\times G\times \BG_m}\left(V^{\mu_1}\otimes V^{\mu_2}
\otimes \BC^d, (\Vl\underset{\fz^{\reg,\la}_\fg}\otimes
\BV_{\fg,\crit}^{\tau(\la)}) \underset{\fz^{\reg,\la}_\fg}\otimes
\on{Sym}_{\fz^{\reg,\la}_\fg}
(\CN^\la_{\reg/\on{unr}})\right)\right).
\end{align*}

Since $\CN^\la_{\reg/\on{unr}}\simeq \Omega^1(\fz^{\reg,\la}_\fg)$,
and since each $\fz^{\reg,\la}_\fg$ is isomorphic to a polynomial
algebra, the multiplicities of the irreducibles in the $G\times
G\times \BG_m$-modules
$$(\Vl\underset{\fz^{\reg,\la}_\fg}\otimes \BV_{\fg,\crit}^{\tau(\la)})
\underset{\fz^{\reg,\la}_\fg}\otimes \on{Sym}_{\fz^{\reg,\la}_\fg}
(\CN^\la_{\reg/\on{unr}}) \text{ and }
\Vl\otimes \BV_{\fg,\crit}^{\tau(\la)}$$
are the same.

\medskip

Hence, it suffices  to show that for each $\mu_1,\mu_2,d$,
\begin{align*}
&\dim\left(\Hom_{G\times G\times \BG_m}(V^{\mu_1}\otimes V^{\mu_2}
\otimes \BC^d, \fD_{G,\crit}^{\ch})\right)= \\ &\underset{\la}\sum\,
\dim\left(\Hom_{G\times G\times \BG_m}\left(V^{\mu_1}\otimes V^{\mu_2}
\otimes \BC^d, \Vl\otimes \BV_{\fg,\crit}^{\tau(\la)}\right)\right).
\end{align*}

However, the one-parameter families of $G\times G\times \BG_m$-modules
given by $\BV_{\fg,\kappa_\hslash+\kappa_\crit}^{\la}\otimes
\BV_{\fg,-\kappa_\hslash+\kappa_\crit}^{\tau(\la)}$ and
$\fD_{G,\hslash}^{\ch}$, where $\kappa_{\hslash} = \hslash \ka_0$ for
some non-zero invariant inner product $\ka_0$, are
$\hslash$-flat. Hence, it is sufficient to check the equality
\begin{multline*}
\dim\left(\Hom_{G\times G\times \BG_m}(V^{\mu_1}\otimes V^{\mu_2}
\otimes \BC^d, \fD_{G,\hslash}^{\ch})\right)= \\
=\underset{\la}\sum\, \dim\left(\Hom_{G\times G\times
\BG_m}\left(V^{\mu_1}\otimes V^{\mu_2} \otimes \BC^d,
\BV_{\fg,\kappa_\hslash+\kappa_\crit}^{\la} \otimes
\BV_{\fg,-\kappa_\hslash+\kappa_\crit}^{\tau(\la)}\right)\right)
\end{multline*}
for a generic $\hslash$. The latter equality indeed holds, since for
$\hslash$ irrational we have an isomorphism of $\hg_\crit$-bimodules:
$$\fD_{G,\hslash}^{\ch}\simeq \underset{\la}\bigoplus\, 
\BV_{\fg,\kappa_\hslash+\kappa_\crit}^{\la}\otimes 
\BV_{\fg,-\kappa_\hslash+\kappa_\crit}^{\tau(\la)}$$
by \lemref{generic}.

\section{The category of spherical modules}

In this section we use the results on the CADO obtained in the
previous section to prove the Main Theorem stated in the
Introduction.

\ssec{Semi-infinite cohomology functor}

Define the character
\begin{equation}    \label{Psi0}
\chi_0: \n_+\ppart \to \C
\end{equation}
by the formula
$$
\chi_0(e_{\al,n}) = \begin{cases} 1, & \on{if} \al = \al_\imath, n=-1,
  \\ 0, & \on{otherwise}. \end{cases}
$$
We have the functors of semi-infinite cohomology (the $+$ quantum
Drinfeld--Sokolov reduction) from the category of
$\ghat_\crit$-modules to the category of graded vector spaces,
\begin{equation}    \label{n+}
{\mc M} \mapsto H^{\frac{\infty}{2}+i}(\fn_+\ppart,\fn_+[[t]],{\mc M}
\otimes \chi_0),
\end{equation}
introduced in \cite{FF,FKW} (see also \cite{vertex}, Ch. 15, and
\cite{FG2}, Sect. 18; we follow the notation of the latter).

More generally, for a complex $\CM^\bullet$ of $\hg_\crit$-modules (or
of $\fn\ppart$-modules), the corresponding semi-infinite Chevalley
complex
$$\bC^\semiinf(\fn\ppart,\CM^\bullet\otimes \chi_0)$$
gives rise to a well-defined triangulated functor
$$D^+(\hg_\crit\mod)\to D(\Vect).$$ This is an analogue of the Whittaker
functor in representation theory of reductive groups over local
fields.

Since $\fZ_\fg$ maps to the center of the category $\hg_\crit\mod$,
the above functor naturally lifts to a functor 
$$D^+(\hg_\crit\mod)\to D(\fZ_\fg\mod).$$ 
By \thmref{support}, the composed functor
$$D^+(\hg_\crit\mod^{G[[t]]})\to D^+(\hg_\crit\mod)\to D(\fZ_\fg\mod),$$
factors through a functor
$$\Psi:D^+(\hg_\crit\mod^{G[[t]]})\to
D\left(\QCoh(\Spec(\fZ_\fg^{\on{unr}}))\right).$$

\medskip

The main result of this paper is the following:

\begin{thm}  \label{main}
The functor $\Psi$ is exact (with respect to the natural t-structures)
and defines an equivalence of abelian categories
\begin{equation} \label{equiv}
\hg_\crit\mod^{G[[t]]} \overset\sim{\longrightarrow}
\QCoh(\Spec(\fZ_\fg^{\on{unr}})).
\end{equation}
\end{thm}

\ssec{Strategy of the proof}

We will derive \thmref{main} from the following two statements.

\begin{prop}  \label{cor FT}
There exists an isomorphism of algebras
\begin{equation} \label{FT}
R^\bullet\Hom_{\hg_\crit\mod^{G[[t]]}}(\Vl,\Vl)\simeq
R^\bullet\Hom_{\QCoh(\Spec(\fZ_\fg^{\on{unr}}))}
(\fz^{\reg,\la}_\fg,\fz^{\reg,\la}_\fg).
\end{equation}
\end{prop}

Let $\hg_\crit\mod_{\reg,\la}$ be the full subcategory of
$\hg_\crit\mod$, consisting of modules, whose support over $\fZ_\fg$
is contained in $\Spec(\fz^{\reg,\la}_\fg)$. Let
$\hg_\crit\mod_{\reg,\la}^{G[[t]]}$ denote the intersection
$$\hg_\crit\mod_{\reg,\la} \cap \hg_\crit\mod^{G[[t]]}.$$

\begin{prop} \label{reg equiv}
The functors $\Psi$ and $\CL\mapsto
\Vl\underset{\fz^{\reg,\la}_\fg}\otimes \CL$ define mutually
quasi-inverse equivalences
$$\hg_\crit\mod_{\reg,\la}^{G[[t]]}\leftrightarrows
\fz^{\reg,\la}_\fg\mod.$$
\end{prop}

For $\la = 0$ this was proved in \cite{FG1} (see also Conjecture
10.3.12 of \cite{F:book}).

We remark that, conversely, both of these propositions follow from
\thmref{main} and the isomorphism
\begin{equation} \label{semiinf of vac}
\Psi(\Vl)\simeq \fz^{\reg,\la}_\fg
\end{equation}
established in \cite{FG6}.

Note that $\Spec(\fz^{\reg,\la}_\fg)\to \Spec(\fZ_\fg^{\on{unr}})$
is a regular embedding. Therefore we have
$$R^\bullet\Hom_{\QCoh(\Spec(\fZ_\fg^{\on{unr}}))}
(\fz^{\reg,\la}_\fg,\fz^{\reg,\la}_\fg) \simeq
\Lambda^\bullet_{\fz^{\reg,\la}_\fg} (\CN^\la_{\reg/\on{unr}}).$$

Combining this with \propref{cor FT}, we obtain:

\begin{cor}
\begin{equation} \label{true FT}
R^\bullet\Hom_{\hg_\crit\mod^{G[[t]]}}(\Vl,\Vl)\simeq
\Lambda^\bullet_{\fz^{\reg,\la}_\fg}
(\CN^\la_{\reg/\on{unr}}).
\end{equation}
\end{cor}

\medskip

For $\la=0$ the isomorphism \eqref{true FT} was established in
\cite{FT} by other methods. The above proof is independent of
\cite{FT} and therefore provides an alternative argument.

\medskip

Let $\chi$ be a $\C$-point of $\Op^{\on{unr}}$, that is, a
$\la$-regular oper in $\Op^{\reg,\la}$ for some $\la \in
\Lambda^+$. We denote by $\gmod_\chi$ the category of
$\hg_\crit$-modules on which the center $\fZ_\fg$ acts according to
the character associated to $\chi$. Let $\gmod_\chi^{G[[t]]}$ be the
corresponding $G[[t]]$-equivariant category. This category contains
the quotient $\Vl(\chi)$ of the Weyl module $\Vl$ by the central
character $\chi$. \thmref{main} then has the following corollary (see
Conjecture 10.3.11 of \cite{F:book}):

\begin{cor}
For any $\chi \in \Op^{\reg,\la}, \la \in \Lambda^+$, the category
$\gmod^{G[[t]]}_\chi$ is equivalent to the category of vector spaces:
its unique, up to isomorphism, irreducible object is $\Vl(\chi)$ and
any other object is isomorphic to a direct sum of copies of
$\Vl(\chi)$. This equivalence is given by the functor $\Psi$.
\end{cor}

This provides a non-trivial test of our conjecture, described in the
Introduction (see formula \eqref{indep of oper}), that the categories
$\gmod_{\chi}^K$ and $\gmod_{\chi'}^K$ are equivalent whenever the
local systems underlying $\chi$ and $\chi'$ are isomorphic to each
other.

\ssec{Computation of $\Psi$}

The first step is to compute the functor $\Psi$ on the objects
$\fD_{G,\crit}^{\ch,\la}$. Since the functor $\Psi$ commutes with
direct limits, from \thmref{structure} and \eqref{semiinf of vac} we 
obtain that $\CB_G^\la:=\Psi(\fD_{G,\crit}^{\ch,\la})$ is acyclic off
cohomological degree $0$ (here we view $\fD_{G,\crit}^{\ch,\la}$
as an object of $\hg_\crit\mod^{G[[t]]}$ via the left action $\fl$). 

%

\medskip

\begin{prop}  \label{calc on D}
The functor $\Psi$ defines an isomorphism
$$\Hom_{\hg_\crit\mod}(\Vl,\fD_{G,\crit}^{\ch,\la})\to
\Hom_{\fZ\mod}(\fz^{\reg,\la}_\fg,\CB_G^\la)$$ (here we consider the
left action $\fl$ of $\hg_\crit$ on
$\fD_{G,\crit}^{\ch,\la}$). Furthermore, the higher $R^i\Hom$'s,
$$R^i\Hom_{\hg_\crit\mod^{G[[t]]}}(\Vl,\fD_{G,\crit}^{\ch,\la})
\quad \text{ and } \quad
R^i\Hom_{D(\QCoh(\Spec(\fZ_\fg^{\on{unr}})))}(\fz^{\reg,\la}_\fg,
\CB_G^\la),$$
vanish.
\end{prop}

\begin{proof}
From \lemref{inv in Weyl}, we know that
$$R^i\Hom_{\hg_\crit\mod^{G[[t]]}}(\Vl,\fD_{G,\crit}^{\ch,\la})=0$$
for $i>0$ and
$$\Hom_{\hg_\crit\mod}(\Vl,\fD_{G,\crit}^{\ch,\la})\simeq 
\BV_{\fg,\crit}^{\tau(\la)}.$$

\medskip

By \thmref{structure}(2) and formula \eqref{semiinf of vac},
$\CB^\la_G$ has a filtration with the associated
graded quotients given by
$$\on{gr}^n(\CB_G^\la)\simeq \BV_{\fg,\crit}^{\tau(\la)}
\underset{\fz^{\reg,\la}_\fg}\otimes
\on{Sym}^n_{\fz^{\reg,\la}_\fg}(\CN^\la_{\reg/\on{unr}}).$$ Moreover,
it follows from the definition of the filtration on
$\fD_{G,\crit}^{\ch}$ that this filtration is the canonical one, given
by the powers of annihilation by $\bI^\la$. This implies that
$R^i(\iota^\la)^!(\CB_G^\cla)=0$ for $i>0$ and that the natural map
$$\BV_{\fg,\crit}^{\tau(\la)}\to R^0(\iota^\la)^!(\CB_G^\cla)$$
is an isomorphism. 

\end{proof}
 

\ssec{Proof of \propref{cor FT}}

Consider the relative Chevalley complex
$$\bC^\bullet(\fg[[t]];\g,\fD_{G,\crit}^{\ch,\la}\otimes V^\la)$$ 
taken with respect to the {\it right} action of $\hg_\crit$
on $\fD_{G,\crit}^{\ch,\la}$, as
a complex of objects of $\hg_\crit\mod^{G[[t]]}$. By \lemref{inv in
Weyl}, it is quasi-isomorphic to $\Vl$ itself. We need to show that
the functor $\Psi$ induces isomorphisms
\begin{multline*}
R^i\Hom_{\hg_\crit\mod^{G[[t]]}}
\left(\Vl,\bC^\bullet(\fg[[t]];\g,\fD_{G,\crit}^{\ch,\la}\otimes
V^\la)\right)\to \\ \to
R^i\Hom_{D(\QCoh(\Spec(\fZ_\fg^{\on{unr}})))}
\left(\fz^{\reg,\la}_\fg,\Psi\left(\bC^\bullet(\fg[[t]];
\g,\fD_{G,\crit}^{\ch,\la}\otimes V^\la) \right)\right).
\end{multline*}

\medskip

Taking into account \propref{calc on D}, it remains to show that the
natural map
\begin{equation} \label{bi-complex}
\Psi\left(\bC^\bullet(\fg[[t]];\g,\fD_{G,\crit}^{\ch,\la}\otimes
V^\la)\right)\to \bC^\bullet(\fg[[t]];\g,\CB^\la_G\otimes V^\la)
\end{equation}
is an isomorphism, i.e., that the corresponding spectral sequences 
converges. 

\medskip

The latter is established as follows: we endow the bi-complex in the
LHS of \eqref{bi-complex} with an additional $\BZ$-grading, as in
\cite{FG6}, Section 4 (see also \cite{FG2}, Section 18.11). We obtain
that in each graded degree, the corresponding bi-complex is
concentrated in a shift of a positive quadrant, hence the convergence.

\ssec{Proof of \propref{reg equiv}}

We are now ready to derive \propref{reg equiv}. The fact that functor
$$\fz^{\reg,\la}_\fg\mod\to \hg_\crit\mod_{\reg,\la}^{G[[t]]},$$
given by 
$$\CL\mapsto \Vl\underset{\fz^{\reg,\la}_\fg}\otimes \CL,$$ is an
equivalence, follows from \propref{cor FT} by repeating verbatim the
argument in \cite{FG1}, Section 8.

\medskip

It remains to show that $\Psi(\Vl\underset{\fz^{\reg,\la}_\fg}\otimes
\CL)$ is acyclic away from the cohomological degree $0$, and that the
$0$-th cohomology is isomorphic to $\CL$.

\medskip

Since the functors appearing above commute with direct limits, we can
assume that $\CL$ is finitely presented. Since $\fz^{\reg,\la}_\fg$ is
isomorphic to a polynomial algebra, we can further assume that $\CL$
admits a finite resolution by free $\fz^{\reg,\la}_\fg$-modules.  This
reduces the assertion to the formula \eqref{semiinf of vac}.

\ssec{Exactness}

We are now ready to show that the functor $\Psi$ is exact, i.e.,
that for $\CM\in \hg_\crit\mod^{G[[t]]}$, the object $\Psi(\CM)$
is acyclic away from the cohomological degree $0$. 

\medskip

Indeed, since $\Psi$ commutes with direct limits, we can assume that
$\CM$ is supported at the $k$-th infinitesimal neighborhood of
$\Spec(\fz_\fg^{\reg,\la})$ inside $\Spec(\fZ_\fg^{\on{unr},\la})$.
By (finite) devissage, that is, by representing $\CM$ as a
$k$-iterated successive extension of modules supported at
$\Spec(\fz_\fg^{\reg,\la})$, we may further assume that $\CM$ belongs
to $\hg_\crit\mod_{\reg,\la}^{G[[t]]}$. In the latter case, the
assertion follows from \propref{reg equiv}.

\ssec{Completion of the proof of \thmref{main}}

Let us now show that the functor $\Psi$ induces isomorphisms
\begin{equation} \label{FT arb}
R^i\Hom_{\hg_\crit\mod^{G[[t]]}}
\left(\Vl,\CM\right)\to
R^i\Hom_{D(\QCoh(\Spec(\fZ_\fg^{\on{unr}})))}
\left(\fz^{\reg,\la}_\fg,\Psi(\CM)\right)
\end{equation}
for any $i$ and $\CM\in \hg_\crit\mod^{G[[t]]}$.

\medskip

Both sides commute with direct limits in $\CM$, so we can again assume
that $\CM$ is supported at the $k$-th infinitesimal neighborhood of
$\Spec(\fz_\fg^{\reg,\la})$ inside $\Spec(\fZ_\fg^{\on{unr},\la})$,
and further that it is on object of
$\hg_\crit\mod_{\reg,\la}^{G[[t]]}$, i.e.,
$$\CM\simeq \Vl\underset{\fz^{\reg,\la}_\fg}\otimes \CL$$ for some
$\fz_\fg^{\reg,\la}$-module $\CL$. Using commutation with direct
limits again, we can assume that $\CL$ is finitely presented, and
hence admits a finite resolution by free $\fz_\fg^{\reg,\la}$-modules.
In the latter case, the isomorphism of \eqref{FT arb} follows from
\propref{cor FT}.


By the same devissage procedure we conclude that the functor
$\Psi$ induces isomorphisms 
$$R^i\Hom_{\hg_\crit\mod^{G[[t]]}}(\CM_1,\CM_2)\to
R^i\Hom_{D(\QCoh(\Spec(\fZ_\fg^{\on{unr}})))}(\Psi(\CM_1),\Psi(\CM_2))$$
for any $i$ and $\CM_1,\CM_2\in \hg_\crit\mod^{G[[t]]}$.

\medskip

Finally, it remains to see that $\Psi$ is essentially
surjective. Again, by commutation with direct limits, it is sufficient
to see that any $\CL\in \QCoh(\Spec(\fZ_\fg^{\on{unr}}))$ supported
at the $k$-th infinitesimal neighborhood of
$\Spec(\fz_\fg^{\reg,\la})$ lies in the image of $\Psi$.

\medskip

Since $\Psi$ induces an isomorphism on the level of $\Ext^1$, by
induction, we can assume that $k=0$, i.e., $\CL\in
\fz_\fg^{\reg,\la}\mod$. In the latter case, the assertion follows
from \propref{reg equiv}.

\end{document}